\documentclass{amsproc}
\usepackage{amsbsy}
\usepackage{amsfonts}
\usepackage{amsmath}
\usepackage{amsthm}
\usepackage{amssymb}
\usepackage{enumitem}
\usepackage{hyperref}
\begin{document}

\title{A central limit theorem for the zeroes of the zeta function}
\author{Brad Rodgers}
\date{}
\thanks{Research supported in part by an NSF RTG Grant.}
\address{Department of Mathematics, UCLA, Los Angeles CA 90095-1555, USA}
\email{brodgers@math.ucla.edu}

\maketitle

\newenvironment{nmath}{\begin{center}\begin{math}}{\end{math}\end{center}}

\newtheorem{thm}{Theorem}[section]
\newtheorem{lem}[thm]{Lemma}
\newtheorem{prop}[thm]{Proposition}
\newtheorem{cor}[thm]{Corollary}
\newtheorem{conj}[thm]{Conjecture}
\newtheorem{dfn}[thm]{Definition}

\newcommand\MOD{\textrm{ (mod }}
\newcommand\s{\textrm{ }}
\newcommand\Var{\mathrm{Var}}
\newcommand\E{\mathbb{E}}
\newcommand\Tr{\mathrm{Tr}}

\begin{abstract}
On the assumption of the Riemann hypothesis, we generalize a central limit theorem of Fujii regarding the number of zeroes of Riemann's zeta function that lie in a mesoscopic interval. The result mirrors results of Spohn and Soshnikov and others in random matrix theory. In an appendix we put forward some general theorems regarding our knowledge of the zeta zeroes in the mesoscopic regime.
\end{abstract}

\section{Introduction}
This paper is an account of a mesoscopic central limit theorem for the number of zeroes of the Riemann zeta function as counted by a (possibly) smoothed counting function. We define the term `mesoscopic' below.  We assume the Riemann hypothesis (RH) throughout the note. On RH, the zeroes of the Riemann zeta function may be labeled $\tfrac{1}{2}+i\gamma$, where $\gamma$ is real. As is customary, we sometimes refer the $\gamma$'s themselves as zeroes, at least where there is no confusion caused. Our concern is the statistical distribution of $\gamma$ near some large (random) height $T$.

\textit{Notation:} We will follow the conventions that $e(x) = e^{i2\pi x}$, the Fourier transform of a function is $\hat{f}(\xi) = \int e(-x \cdot \xi) f(x)\, dx$, and the inverse Fourier transform is $\check{g}(x) = \int e(x\cdot \xi)g(\xi) \,d\xi$. In addition we use the notations $|f(x)| \lesssim g(x)$ and $f(x) = O(g(x))$ interchangeably to mean there is a constant $C$ not depending on $x$ so that $|f(x)| \leq C g(x)$. Finally, in cases where the context is clear, we sometimes use the abbreviation $K_L(x) = K(x/L)$.

If $N(T)$ is the number of nontrivial zeroes in the upper half plane with height no more than $T$, then the number of zeroes $N(t+h)-N(t)$ to occur in an interval $[t,t+h]$ is expected to be roughly $h \tfrac{\log t}{2\pi}$ \cite{Ti}. It was first shown by Fujii \cite{Fu} that the oscillation of this quantity is Gaussian, with a variance depending upon the number of zeroes expected to lie in the interval.

\begin{thm}[Fujii's mesoscopic central limit theorem]
\label{Fujiimeso}
Let $n(T)$ be a fixed function tending to infinity as $T\rightarrow\infty$ in such a way that $n(T) = o(\log T)$, and let $t$ be a random variable uniformly distributed on the interval $[T,2T]$. For notational reasons we label by $X_T$ the probability space from which the random variable $t$ is drawn. Then, letting $\Delta = \Delta(t,T) := N(t + \tfrac{2\pi n(T)}{\log T}) - N(t)$,
$$
\E_{X_T} \Delta = n(T) + o(1),
$$
$$
\Var_{X_T}(\Delta) \sim \frac{1}{\pi^2}\log n(T),
$$
and in distribution
$$
\frac{\Delta-\E \Delta}{\sqrt{\Var \Delta}} \Rightarrow N(0,1)
$$
as $T\rightarrow\infty$.
\end{thm}

The main purpose of this note is to generalize Fujii's theorem in the following way:

\begin{thm}[A general mesoscopic central limit theorem]
\label{generalmeso}
Let $n(T)$ and $X_T$ be as in Theorem \ref{Fujiimeso}. For a fixed real valued function $\eta$ with compact support and bounded variation, define
$$
\Delta_\eta = \Delta_\eta(t,T) = \sum_\gamma \eta\big(\tfrac{\log T}{2\pi n(T)} (\gamma-t)\big),
$$
where the sum is over all zeros $\gamma$, counted with multiplicity. In the case that $\int |x||\hat{\eta}(x)|^2 \, dx$ diverges, we have
$$
\E_{X_T} \Delta_\eta = n(T)\int_\mathbb{R}\eta(\xi) d\xi + o(1),
$$
$$
\Var_{X_T}(\Delta_\eta) \sim \int_{-n(T)}^{n(T)} |x| |\hat{\eta}(x)|^2 dx
$$
and in distribution
$$
\frac{\Delta_\eta -\E \Delta_\eta}{\sqrt{\Var \Delta_\eta}} \Rightarrow N(0,1)
$$
as $T\rightarrow\infty$.
\end{thm}

It is a straightforward computation to see that Theorem \ref{Fujiimeso} follows from Theorem \ref{generalmeso} by letting $\eta = \mathbf{1}_{[-1/2,1/2]}.$

Additionally, in the case of variances that converge:

\begin{thm}
\label{generalmesoconv}
For $n(T)$ and $X_T$ as in Theorem \ref{generalmeso}, but $\eta$ with compact support and bounded second derivative, the integral $\int |x||\hat{\eta}(x)|^2 \, dx$ is necessarily finite, but the conclusion of Theorem \ref{generalmeso} remains true even still.
\end{thm}

\textit{Remark:} The condition that the test functions in Theorems \ref{generalmeso} and \ref{generalmesoconv} have compact support is not optimised; somewhat looser decay properties, along the lines of quadratic decay, are sufficient for the proof that follows.

We call Theorems \ref{Fujiimeso}, \ref{generalmeso}, and \ref{generalmesoconv} `mesoscopic' central limit theorems as they concern collections of $n(T)$ zeroes which grow to infinity, but intervals whose length $\tfrac{2\pi n(T)}{\log T}$ tends to $0$ all the same.

On such mesoscopic intervals (averaged as in Theorems \ref{Fujiimeso} and \ref{generalmeso}), all evidence points to the zeroes resembling points in a determinantal point process with sine kernel. In the microscopic regime (when $n(T) = O(1)$) this is known to be the case, provided we restrict our attention to the statistics counted by sufficiently smooth test functions. (See Rudnick and Sarnak \cite{RuSa} or Hughes and Rudnick \cite{HuRu}.) The techniques that follow allow us to recover these results for smooth test functions, as well as extend them to a mesoscopic regime, in a sense to be specified. These matters are discussed in the appendix.

For the moment, we may simply note the similarity of Theorems \ref{Fujiimeso}, \ref{generalmeso}, and \ref{generalmesoconv} to certain results in the theories of random matrices and determinantal point processes (for an introduction to the latter, see \cite{HoKrPeVi} or the introduction of \cite{So1}):

\begin{thm}[Costin and Lebowitz]
Let $X$ be a determinantal point process on $\mathbb{R}$ with sine kernel $K(x,y) = \tfrac{\sin \pi(x-y)}{\pi(x-y)}$, and $\Delta$ a count of the number of points lying in the interval $[0,L]$. Then
$$
\E_X \Delta = L,
$$
$$
\Var_X(\Delta) \sim \frac{1}{\pi^2}\log L
$$
and in distribution
$$
\frac{\Delta-\E \Delta}{\sqrt{\Var \Delta}} \Rightarrow N(0,1)
$$
as $L \rightarrow \infty$.
\end{thm}

For more general test functions, the analogue for determinantal point processes of Theorem \ref{generalmesoconv} is a corollary of both Theorem 3 or 4 of Soshnikov \cite{So1}, who attributes this corollary to Spohn \cite{Sp}. Proofs of closely related results can also be found in \cite{So2}.

\begin{thm}
Let $X$ be a determinantal point process on $\mathbb{R}$ with sine kernel $K(x,y) = \tfrac{\sin \pi(x-y)}{\pi(x-y)}$. For $\eta$ a Schwartz function and $L$ a positive number, define
$$
\Delta_\eta = \sum \eta(x_i/L)
$$
where $((x_i))$ are the points of the point process. Note that for such $\eta$, $|x||\hat{\eta}(x)|^2$ is always integrable. We have
$$
\E_{X} \Delta_\eta = L\int_\mathbb{R}\eta(\xi) d\xi,
$$
$$
\Var_{X}(\Delta_\eta) \sim \int_{-\infty}^\infty |x| |\hat{\eta}(x)|^2 dx
$$
and in distribution
$$
\frac{\Delta_\eta -\E \Delta_\eta}{\sqrt{\Var \Delta_\eta}} \Rightarrow N(0,1)
$$
as $L\rightarrow\infty$.
\end{thm}

Note that in this case there is no limit on the growth rate of $L$, which is not surprising since the theorem is for a single point process $X$ rather than a series of point processes.

\textit{Remark:} The test functions for this theorem are Schwartz class. One expects the theorem to be true for any test functions $\eta$ for which $|x||\hat{\eta}(x)|^2$ is integrable, a considerably larger class. A statement of this does not appear to be in the literature for the sine-kernel determinantal point process, but analogous results of this sort are known for the point processes induced by eigenvalues of random $n \times n$ unitary matrices, and test functions that count all $n$ eigenvalues. (One may call such theorems `macroscopic' as opposed to mesoscopic theorems for test functions that count only $o(n)$ eigenvalues.) Such results are known as strong Szeg\H{o} theorems, and considerable literature surrounds the subject. (See \cite{Si}, Ch. 6 for instance.)

Likewise, the analogue of Theorem \ref{generalmeso} does not seem to be in the literature, but an analogue was proved by Diaconis and Evans \cite{DiEv} for counts of all $n$ eigenvalues of unitary matrices, among other ensembles. The perspective of Diaconis and Evans is perhaps most similar to ours here.

We should therefore expect that Theorem \ref{generalmeso} is true even in the case that $\int |x||\hat{\eta}(x)|^2\,dx$ converges with no more restrictions on $\eta$ than a bound on variation -- this would encompass Theorem \ref{generalmesoconv} -- but in the latter theorem we require not only that this integral converge, but that it converge somewhat rapidly. Bounding an error term prevents us from accessing the results in between the two theorems, even though by analogy we should fully expect them to be true.

In fact, Fujii proved a more general result than Theorem \ref{Fujiimeso}, encompassing macroscopic intervals as well. In order to state Fujii's result succinctly, we recall the definition
$$
S(t) := \arg \zeta(\tfrac{1}{2}+it),
$$
where argument is defined by a continuous rectangular path from $2$ to $2+it$ to $\tfrac{1}{2}+it$, beginning with $\arg 2 = 0$, and by upper semicontinuity in case this path passes through a zero. $S(t)$, as it ends up, is small and oscillatory, and our interest in it derives from the fact that it appears as an error term in the zero counting function:
\begin{equation}
\label{zerocount}
N(T) = \tfrac{1}{\pi}\arg \Gamma\big(\tfrac{1}{4}+i\tfrac{T}{2}\big) - \tfrac{T}{2\pi}\log\pi + 1 + S(T).
\end{equation}

\begin{thm} [Fujii's macroscopic central limit theorem]
\label{Fujiimacro}
Let $X_T$ be as in Theorem \ref{Fujiimeso}, and $n(T)$ with $\log T \lesssim n(T) \lesssim T$. Define $\tilde{\Delta} = S(t + \tfrac{2\pi n(T)}{\log T}) - S(t)$. Then
$$
\E_{X_T} \tilde{\Delta} = o(1),
$$
$$
\Var_{X_T}(\tilde{\Delta}) \sim \frac{1}{\pi^2}\log \log T,
$$
and in distribution
$$
\frac{\tilde{\Delta}}{\sqrt{\Var \tilde{\Delta}}} \Rightarrow N(0,1)
$$
as $T\rightarrow\infty$.
\end{thm}

Note that in this case, if $\Delta$ is defined as before with respect to the function $N(t)$, $\E_{X_T} \Delta$ does not have quite as nice an expression owing to the growth of the logarithm function.

In fact, it will in general prove preferable to work with $S(t)$ in place of $N(t)$ in the computations that follow. Differentiating \eqref{zerocount}, we have
$$
\big[\tilde{d}(\xi)-\tfrac{\Omega(\xi)}{2\pi}\big]\,d\xi = dS(\xi),
$$
where
$$
\tilde{d}(\xi) := \sum_{\gamma}\delta(\xi-\gamma),
$$
with the sum over zeroes counted with multiplicity, and
$$
\Omega(\xi) := \tfrac{1}{2}\tfrac{\Gamma'}{\Gamma}\big(\tfrac{1}{4}+i\tfrac{\xi}{2}\big) +\tfrac{1}{2}\tfrac{\Gamma'}{\Gamma}\big(\tfrac{1}{4}-i\tfrac{\xi}{2}\big) - \log \pi.
$$

Making use of the moment method\footnote{An introduction to the moment method can be found in for example \cite{Ta} section 2.2.3.}  and Stirling's formula \footnote{Stirling's formula (proved in \cite{AnAsRo} section 1.4 for instance) implies that $\frac{\Omega(\xi)}{2\pi} = \frac{\log\big((|\xi|+2)/2\pi\big)}{2\pi} + O\Big(\frac{1}{|\xi|+2}\Big)$.}, we see that Theorem \ref{generalmeso} will be implied by
\begin{thm}
\label{smoothedmeso}
For $\eta$ a real-valued function with compact support and bounded variation, for $n(T)\rightarrow\infty$ as $T\rightarrow\infty$ in such a way that $n(T) = o(\log T)$,
$$
\frac{1}{T}\int_T^{2T}\bigg[\int_\mathbb{R} \eta\big(\tfrac{\log T}{2\pi n(T)}(\xi-t)\big)\,dS(\xi)\bigg]^k \,dt = (c_k+o(1))\bigg[\int_{-n(T)}^{n(T)}|x||\hat{\eta}(x)|^2\,dx \bigg]^{k/2},
$$
provided the integral on the right diverges. Here $c_{\ell} := (\ell-1)!!$ for even $\ell$, and $c_{\ell} := 0$ for odd $\ell$, are the moments of a standard normal variable.
\end{thm}

Theorem \ref{generalmesoconv}, as well, will follow from the above statement where $\eta$ is instead restricted as in Theorem \ref{generalmesoconv}.

In order to prove his results, Fujii made use of the moment method, and the following (unconditional) approximation due to Selberg \cite{Sel1},\cite{Sel2},
\begin{equation}
\label{selberg}
\frac{1}{T}\int_T^{2T}\Big[S(t) + \tfrac{1}{\pi}\sum_{p\leq T^{1/k}}\frac{\sin(t \log p)}{\sqrt{p}}\Big]^{2k} \,dt = O(1),
\end{equation}
which Selberg had used earlier to derive a more global central limit theorem for $S(t)$,
$$
\frac{1}{T}\int_T^{2T}|S(t)|^{2k}\,dt \sim \tfrac{(2k-1)!!}{(2\pi^2)^k}(\log\log T)^k.
$$

These formulas are sufficient to prove Theorem \ref{generalmeso} for test functions $\eta$ which are sums of a finite number of indicator functions. They break down, however, in an attempt to prove the theorem for general $\eta$, since, although one can approximate $\eta$ by simple functions, the error terms thus generated rapidly overwhelm the main terms of the moments.

We do not therefore make use of Selberg's approximation for $S(t)$, and indeed do not pass through his well known mollification for $\zeta'/\zeta$, which Selberg and many other authors typically make use of to prove statistical theorems about zeta zeroes.

Our approach is outlined in the next section. Very roughly stated, it is a sort of weak analogue of the approach of Selberg and Fujii. In this, we follow the derivation \cite{HuRu} of Hughes and Rudnick of mock-gaussian behavior in the microscopic regime with respect to sufficiently smooth test functions. We extend these computations to the mesoscopic regime, still requiring smoothness, but the key point which allows us to obtain our central limit theorem is that any test function will become sufficiently smooth when dilated as they are in the central limit theorems \ref{generalmeso} and \ref{generalmesoconv}. This is one clarifying feature of our proof. The proof of Fujii's theorem making use of Selberg's approximation for $S(t)$ leaves the link between this central limit theorem and the microscopic determinantal structure of the zeroes somewhat mysterious.

This approach, with slightly more work, can be used to produce Fujii's Theorem \ref{Fujiimacro} as well, although in this case an analogue of Theorem \ref{generalmeso} is less satisfying. We shall not prove so in this note, but in the macroscopic case already if $\eta$ is so much as absolutely continuous, the variance and higher moments of $\tilde{\Delta}_\eta$ (defined in the obvious way) tend to $0$. This is a feature of the rigidity of the distribution of zeroes at this regime, which while not quite as rigid as a clock distribution (see \cite{KiSt} for a definition), resemble at this level this distribution perhaps somewhat more than they do a sine kernel determinantal point process. One should compare this analogy with the classical theorems that for a fixed $h$, $N(t+h)-N(t) \asymp \log t$ for \emph{all} sufficiently large $t$, with constants depending upon $h$. (See \cite{Ti}, Theorems 9.2 and 9.14.) In this regime, arithmetic factors play a heavy explicit role; this will be implicitly evident in the proof that follows. In this, we can recover the heuristic observations of Berry \cite{Be} regarding the origin for the variance terms in Fujii's theorems. Indeed, one can already discern, by comparing Fujii's central limit theorem to the central limit theorem of Costin and Lebowitz, that the statistics of the zeros in this regime cannot be modeled too closely by a sine-kernel determinantal process. Outside of the mesoscopic regime, these statistics demonstrate an important `resurgence phenomenon' discovered heuristically by Bogomolny and Keating, and explored in \cite{BoKe},\cite{Pe}, \cite{Ro} and \cite{Sn}.

Zeev Rudnick pointed out to the author that he had used similar ideas with Faifman in \cite{FaRu} to prove a Fujii-type central limit theorem, for counting functions with a strict cutoff, in the finite field setting.

One can apply these ideas to get central limit theorems as well for the number of low-lying zeroes of $L(s,\chi_d)$, where $\chi_d$ ranges over the family of primitive quadratic characters, by extending the microscopic statistics of Rubenstein \cite{Rub}.

After a version of this note had been posted online, Theorems \ref{generalmeso} and \ref{generalmesoconv} were independently proven, in a different manner, by Bourgade and Kuan \cite{BoKu}. They proceed by a method which makes use of the Helffer-Sj\"ostrand functional calculus and a mollification formula of Selberg. In brief summary, the dichotomy between the two treatments is that between the use of harmonic analysis (our approach) and complex analysis (the approach of Bourgrade and Kuan).

The conditions on admissible test functions in \cite{BoKu} differ very slightly from those in the theorems of this paper, but not in an important way. Either method seems to include the class of test functions -- lying within the class of test functions with converging variance -- for which a central limit theorem, by analogy with random matrix theory and discussed in the remark above, ought to be true, but for which we have no proof. It would be interesting if other approaches could fill this small but pesky limitation in our knowledge.

\section{A heuristic outline of the proof}

Here we give a heuristic sketch of our approach before proceeding to a rigorous proof.

Instead of Selberg's metric $L^k[0,T]$ approximation for the function $S(t)$, we use the distributional formula,
$$
dS(\xi) = - \int \frac{e^{ix\xi}+e^{-ix\xi}}{2} e^{-x/2}\,d\big(\psi(e^x)-e^x\big).
$$
This formula is to be understood heuristically; some restrictions are entailed on the test functions in $\xi$ against which it can be integrated, a precise statement of which are given by Theorem \ref{explicitform}. The integral on the right has an arithmetic component,
$$
-\frac{1}{\pi}\sum_n \frac{e^{i\xi\log n} + e^{-i\xi \log n}}{2}\frac{\Lambda(n)}{\sqrt{n}} \approx -\frac{1}{\pi} \sum_p \frac{e^{i \xi \log p}+e^{-i \xi \log p}}{2} \frac{\log p}{\sqrt{p}}
$$
(compare with Selberg's \eqref{selberg}), and a continuous component
$$
\int\frac{e^{ix\xi}+e^{-ix\xi}}{2} e^{x/2}\,dx.
$$
It will emerge from computations that the measure in variables $\xi_1,...,\xi_k$, given by
$$
\frac{1}{T}\int_T^{2T} \prod_{\ell=1}^k dS(\xi_\ell+t)\,dt,
$$
is extremely well approximated by substituting for $dS$ in each variable only its arithmetic component:
$$
\frac{1}{T}\int_T^{2T} \prod_{\ell=1}^k\bigg(-\frac{1}{\pi} \sum_p \frac{e^{i (\xi_\ell+t) \log p}+e^{-i (\xi_\ell +t)\log p}}{2} \frac{\log p}{\sqrt{p}}\,d\xi_\ell\bigg)\,dt,
$$
so long as the measures are being integrated against functions $f(\xi_1,...,\xi_k)$ that have their Fourier transform supported at a scale of $O(\log T)$. Said another way, this approximation is a good one so long as the test function $f$ is sufficiently smooth, observed on intervals of size $1/\log T$.

The statistics in which we will be interested for our central limit theorem are
$$
\frac{1}{T}\int_T^{2T} \int_{\xi\in\mathbb{R}^k} \prod_{\ell=1}^k \eta\Big(\tfrac{\log T}{2\pi n(T)} \xi_\ell\Big)\,dS(\xi_\ell+t)\,dt.
$$
For any `nice' function $\eta$, because $n(T)\rightarrow\infty$,  for large enough $T$, the function
$$
\eta\big(\tfrac{\log T}{2\pi n(T)} \xi_\ell\big)\cdots \eta\big(\tfrac{\log T}{2\pi n(T)} \xi_\ell\big)
$$
will be basically smooth in $\xi$ at a scale of $1/\log T$. (One would have to observe at the larger scale of $n(T)/\log T$ to see the variations in these test functions.)  Therefore the above integral can be approximated by
\begin{align*}
&\frac{1}{T}\int_T^{2T} \prod_{\ell=1}^k\int_{-\infty}^\infty \eta\Big(\tfrac{\log T}{2\pi n(T)} \xi_\ell\Big)\bigg(-\frac{1}{\pi} \sum_p \frac{e^{i (\xi_\ell+t) \log p}+e^{-i (\xi_\ell +t)\log p}}{2} \frac{\log p}{\sqrt{p}}\bigg)\,d\xi_\ell\,dt \\
&= \frac{1}{T}\int_T^{2T} \prod_{\ell=1}^k \frac{n(T)}{\log T} \sum_p \frac{\log p}{\sqrt{p}} \Big(\hat{\eta}\Big(\frac{\log p}{\log T/n(T)}\Big) e^{i t\log p} + \hat{\eta}\Big(-\frac{\log p}{\log T/n(T)}\Big)e^{-i t \log p}\Big)\, dt
\end{align*}

For finite collections of primes  $p$ and large $T$, the quantities $e^{i t \log p}$ behave like independent random variables as $t$ ranges over $[T,2T]$. We will be interested in collections of primes $p$ that grow with $T$ (with all primes $p$ so that $\log p / (\log T /n(T)) = O(1)$ in fact, owing to the decay of the function $\hat{\eta}$), but by mimicking the analysis that leads to this observation, we are able to see that the above average contracts to a quantity close to
$$
c_k\, \bigg[\big(\frac{n(T)}{\log T}\Big)^2 \sum_{\log p = O(\tfrac{\log T}{n(T)})} \frac{2 \log^2 p}{p}\hat{\eta}\Big(\frac{\log p}{\log T/n(T)}\Big)\hat{\eta}\Big(-\frac{\log p}{\log T/n(T)}\Big)\bigg]^{k/2}
$$
where $c_k$ are the moments of a standard normal variable. $c_k$ is given also by the number of possible pairings among a set of $k$ elements, and the only terms in this expression that have survived from the expression above it are pairings of equal primes in the expansion of the product inside the integral. Using the prime number theorem, we are able to show that this tends to the right hand side limit in Theorems \ref{generalmeso} and \ref{generalmesoconv}.

There is one point of this proof which deserves further commentary, as it comprises a substantial part of the technical challenge ahead; this is the claim that for any nice function $\eta$, the rescaled function $\eta\big(\tfrac{\log T}{2\pi n(T)} \xi_\ell\big)\cdots \eta\big(\tfrac{\log T}{2\pi n(T)} \xi_\ell\big)$ will be sufficiently smooth at a scale of $1/\log T$. It is certainly not true for an arbitrary function of bounded variation $\eta$ that this rescaling will be locally smooth in the sense we have used above: of having a Fourier transform supported at a scale of $\log T$. For instance, the rescaling of $\eta = \mathbf{1}_{[-1/2,1/2]}$ does not have this property, and indeed no function $\eta$ will unless $\eta$ has compact Fourier transform to begin with. What will be true, however, is that for any function $\eta$ of the sort delimited in Theorems \ref{generalmeso} and \ref{generalmesoconv}, this rescaling can be very well approximated by a function with Fourier transform supported at a scale of $\log T$. Making use of upper bounds for the average number of zeros in an interval of size $1/\log T$, we are able to show that the statistics of this approximation do not deviate much from the statistics of our original test function, and therefore obtain Theorems \ref{generalmeso} and \ref{generalmesoconv}. (Indeed, it is because we must replace test functions with approximations that induce a small error term that we must restrict our attention to a slightly smaller domain of test functions in Theorem \ref{generalmesoconv} than in \ref{generalmeso}.)

\section{Local Limit Theorems for Smooth Test Functions}

This section consists mainly in minor quantitative refinements in the argument of Hughes and Rudnick \cite{HuRu}. In turn, their argument is similar to Selberg's in making use of the fundamental theorem of arithmetic to evaluate certain integrals. Our main tool in what follows will be the well known explicit formula relating the zeroes of the zeta function to the primes.

\begin{thm}
\label{explicitform}
[The explicit formula]
For $g$ a measurable function such that $g(x) = \tfrac{g(x+)+g(x-)}{2}$, and for some $\delta > 0$,
$$
(a)\quad \int_{-\infty}^\infty e^{(\tfrac{1}{2}+\delta)|x|}|g(x)|dx < +\infty,
$$
$$
(b)\quad \int_{-\infty}^\infty e^{(\tfrac{1}{2}+\delta)|x|}|dg(x)| < +\infty,
$$
we have
$$
-\!\!\!\!\!\!\!\int_{-\infty}^\infty \hat{g}\big(\tfrac{\xi}{2\pi}\big) \, dS(\xi) = \int_{-\infty}^\infty [g(x)+g(-x)]e^{-x/2}d\big(e^x-\psi(e^x)\big),
$$
where here $\psi(x) = \sum_{n\leq x} \Lambda(n)$, for the von Mangoldt function $\Lambda$.
\end{thm}

The integral on the left denotes a principle value integral, $\lim_{L\rightarrow\infty}\int_{-L}^L$, and this limit necessarily converges when the conditions of the theorem for $g$ are met. In what follows we will frequently work with test functions for which the distinction between this principle value integral and an ordinary integral disappears, and if this is the case we will cease to make one in notation.

Written in this way, the explicit formula is true only on the Riemann hypothesis. It is due in varying stages to Riemann \cite{Ri}, Guinand \cite{Gu}, and Weil \cite{We}, and expresses a Fourier duality between the error term in the prime number theorem and the error term for of the zero-counting function.

Without the Riemann hypothesis, we must write the left hand side as
$$
\lim_{L\rightarrow\infty}\sum_{|\gamma|<L}\hat{g}\big(\tfrac{\gamma}{2\pi}\big)- \int_{-L}^L \frac{\Omega(\xi)}{2\pi}\hat{g}\big(\tfrac{\xi}{2\pi}\big)\,d\xi
$$
where our sum is over $\gamma$ (possibly complex) such that $\tfrac{1}{2}+i\gamma$ is a nontrivial zero of the zeta function, It is proven by a simple contour integration argument, making use of the the reflection formula to evaluate one-half of the contour. (For a proof, see \cite{Iwa} or \cite{MontVau}.)

We will also need the following corollary of the prime number theorem.

\begin{lem}[A prime number asymptotic] For $f$ with compact support and bounded second derivative,
\label{pnt}
\begin{equation}
\label{1}
\frac{1}{H^2}\sum_{p} \frac{\log^2 p}{p}f\bigg(\frac{\log p}{H}\bigg) = O\bigg(\frac{\|f\|_\infty + \|f'\|_\infty+\|f''\|_\infty}{H^2}\bigg) + \int_0^\infty x f(x) dx.
\end{equation}
\end{lem}

\begin{proof}
That something like this is true is evident from the prime number theorem (or even Chebyshev), but some formal care is required to get the desired error term. We will need that,
$$
\sum_{p \leq n} \frac{\log p}{p} = \log n + C + O(\tfrac{1}{\log^2 n})
$$
for some constant $C$, which is a formula on the level of the prime number theorem (and can be proven from the prime number theorem with a strong error term using partial summation.)

We have then, using the abbreviation $F(x) = x f(x)$,
\begin{align*}
\frac{1}{H^2}\sum_{p} \frac{\log^2 p}{p}f\bigg(\frac{\log p}{H}\bigg) =& \frac{1}{H}\sum_n \bigg[F\big(\tfrac{\log n}{H}\big) - F\big(\tfrac{\log (n+1)}{H}\big)\bigg]\bigg(\log n + C + O\big(\tfrac{1}{\log^2 n}\big)\bigg)\\
=& O\bigg(\frac{\|f\|_\infty + \|f'\|_\infty}{H^2}\bigg) + \sum_n \frac{\log n - \log(n+1)}{H}\cdot \frac{\log n}{H} F'\big(\tfrac{\log n}{H}\big),
\end{align*}
by partial summation and the mean value theorem. Again using the mean value theorem, this time to approximate an integral, we have that this expression is
$$
O\bigg(\frac{\|f\|_\infty + \|f'\|_\infty+\|f''\|_\infty}{H^2}\bigg) + \int_0^\infty x F'(x) dx,
$$
which upon integrating by parts is the right hand side of \eqref{1}.
\end{proof}

In what follows instead of working with the average $\frac{1}{T}\int_T^{2T}$ we work with smooth averages $\int \sigma(t/T)/T$ for bump functions $\sigma$. What we will show is that
\begin{thm}
\label{X}
For $\eta$ as in Theorem \ref{smoothedmeso}, and $\sigma$ non-negative of mass $1$ such that $\hat{\sigma}$ has compact support and $\sigma(t)\log^k(|t|+2)$ is integrable,
$$
\int_\mathbb{R}\frac{\sigma(t/T)}{T}\Big[\int_\mathbb{R} \eta\big(\tfrac{\log T}{2\pi n(T)}(\xi-t)\big)dS(\xi)\Big]^k \,dt = (c_k+o(1))\Big[\int_{-n(T)}^{n(T)}|x||\hat{\eta}(x)|^2\,dx\Big]^{k/2}.
$$
\end{thm}

We will show that this implies Theorem \ref{smoothedmeso} at the end of Section 4. We have a computational lemma.

\begin{lem}
\label{X7}
Suppose we are given non-negative integrable $\sigma$ of mass $1$ such that $\hat{\sigma}$ has compact support, and integrable functions $\eta_1, \eta_2, \dots, \eta_k$ such that $\mathrm{supp}\, \hat{\eta}_\ell \subset [-\delta_\ell,\delta_\ell]$ with $\delta_1+\delta_2 + \cdots + \delta_k = \Delta < 2$. There exists a $T_0$ depending on $\Delta$ and the the region in which $\hat{\sigma}$ is supported so that for $T\geq T_0$,
\begin{align}
\label{compute}
\int_{\mathbb{R}}\frac{\sigma(t/T)}{T}\prod_{\ell=1}^k \Bigg(-\!\!\!\!\!\!\!\int_{-\infty}^\infty\eta_\ell\big(\tfrac{\log T}{2\pi}(\xi_\ell-t)\big)dS(\xi_\ell)\Bigg)dt = &O_k\bigg(\frac{1}{T^{1-\Delta/2}} \prod_{\ell=1}^k \frac{ \|\hat{\eta}_\ell\|_\infty}{\log T}\bigg) \\
&+ \bigg(\frac{-1}{\log T}\bigg)^k \sum_{n_1^{\epsilon_1}n_2^{\epsilon_2}\cdot\cdot\cdot n_k^{\epsilon_k}=1} \prod_{\ell=1}^k\frac{\Lambda(n_\ell)}{\sqrt{n_\ell}} \hat{\eta}_\ell \big(\tfrac{\epsilon_\ell\log n_\ell}{\log T}\big)\notag,
\end{align}
where the sum is over all $n \in \mathbb{N}^k, \epsilon\in \{-1,1\}^k$ such that $n_1^{\epsilon_1}n_2^{\epsilon_2}\cdot\cdot\cdot n_k^{\epsilon_k}=1$.
\end{lem}

\begin{proof}
By the explicit formula, the right hand side of \eqref{compute} is
\begin{align*}
&\int_\mathbb{R}\frac{\sigma(t/T)}{T}\bigg(\prod_{\ell=1}^k \int_{-\infty}^\infty \frac{1}{\log T}\Big[\hat{\eta}_{\ell}\big(-\tfrac{x_\ell}{\log T}\big)e^{-i x_\ell t} + \hat{\eta}_{\ell}\big(\tfrac{x_\ell}{\log T}\big)e^{i x_\ell t}\Big]e^{-x_\ell/2}d\big(e^{x_\ell}-\psi(e^{x_\ell})\big)\bigg)dt\\
&\;\;\;\;\;\;\;= \sum_{\epsilon\in\{-1,1\}^k}\int_{\mathbb{R}^k} \frac{\hat{\sigma}\big(-\tfrac{T}{2\pi} (\epsilon_1 x_1 + \cdot\cdot\cdot + \epsilon_k x_k)\big)}{\log^k T} \prod_{\ell=1}^k \hat{\eta}_\ell\big(\tfrac{\epsilon_\ell x_\ell}{\log T}\big) e^{-x_\ell/2} d\big(e^{x_\ell}-\psi(e^{x_\ell})\big).
\end{align*}
The second line follows by interchanging the order of integration, justified by the compact support of $\hat{\eta}_\ell$. We can expand the product $\prod e^{-x_\ell/2} d\big(e^{x_\ell}-\psi(e^{x_\ell})\big)$ into a sum of signed terms of the sort $d\beta_1(x_1)\cdot\cdot\cdot d\beta_k(x_k)$, where $d\beta_\ell(x)$ is either $e^{x/2}dx$ or $e^{-x/2}d\psi(e^x)$. In the case that at least one $d\beta_j$ in our product is $e^{x/2}dx$ we have
$$
\Bigg|\int_\mathbb{R}\frac{\hat{\sigma}\big(-\tfrac{T}{2\pi} (\epsilon_1 x_1 + \cdot\cdot\cdot + \epsilon_k x_k)\big)}{\log^k T}\hat{\eta}_j\big(\tfrac{\epsilon_j x_j}{\log T}\big) d\beta_j(x_j)\Bigg| \lesssim \frac{\|\hat{\eta}_j\|_\infty}{T\log^k T}T^{\delta_j/2},
$$
so that in this case
\begin{align*}
\Bigg|\int_{\mathbb{R}^k} \frac{\hat{\sigma}\big(-\tfrac{T}{2\pi} (\epsilon_1 x_1 + \cdot\cdot\cdot + \epsilon_k x_k)\big)}{\log^k T} \prod_{\ell=1}^k \hat{\eta}\big(\tfrac{\epsilon_\ell x_\ell}{\log T}\big)d\beta_\ell(x_\ell)\Bigg| &\lesssim \frac{\|\hat{\eta}_j\|_\infty T^{\delta_j/2}}{T\log^k T}\int_{\mathbb{R}^{k-1}}\prod_{\ell\neq j}\hat{\eta}\big(\tfrac{\epsilon_\ell x_\ell}{\log T}\big)d\beta_\ell(x_\ell)\\
&\lesssim \frac{T^{\Delta/2}}{T}\prod_\ell \frac{ \|\hat{\eta}_\ell\|_\infty}{\log T}
\end{align*}
Into such error terms we can absorb all products $d\beta_1\cdot\cdot\cdot d\beta_k$ except that product made exclusively of prime counting measures, namely $(-1)^k\prod e^{-x_\ell/2}d\psi(e^{x_\ell})$. Evaluating the integral of this product measure we have that the left hand side of \eqref{compute} is
$$
O_k\bigg(\frac{1}{T^{1-\Delta/2}} \prod_{\ell=1}^k \frac{ \|\hat{\eta}_\ell\|_\infty}{\log T}\bigg) + \bigg(\frac{-1}{\log T}\bigg)^k\sum_{\epsilon\in\{-1,1\}^k} \sum_{n\in\mathbb{N}^k} \hat{\sigma}\big(-\tfrac{T}{2\pi}(\epsilon_1\log n_1 + \cdot\cdot\cdot + \epsilon_k\log n_k)\big) \prod_{\ell=1}^k\frac{\Lambda(n_\ell)}{\sqrt{n_\ell}} \hat{\eta}_\ell \big(\tfrac{\epsilon_\ell\log n_\ell}{\log T}\big).
$$
Note that if $|\epsilon_1 \log n_1 + \cdot\cdot\cdot + \epsilon_k \log n_k|$ is not $0$, it is greater than $|\log (1-1/\sqrt{n_1\cdot\cdot\cdot n_k})| \geq \frac{\log 2}{\sqrt{n_1\cdot\cdot\cdot n_k}}$ since $n_i$ is always an integer. As $\sqrt{n_1\cdot\cdot\cdot n_k} \leq T^{\Delta/2} = o(T)$ and $\hat{\sigma}$ has compact support, for large enough $T$ our sum is over only those $\epsilon, n$ such that $\epsilon_1 \log n_1 + \cdot\cdot\cdot + \epsilon_k \log n_k = 0$.
\end{proof}

Finally, we can use our prime number asymptotic, Lemma \ref{pnt}, to obtain

\begin{lem}
\label{eval}
For $u_1, ..., u_k$ with bounded second derivative
\begin{equation}
\label{int}
\frac{1}{H^k}\sum_{n_1^{\epsilon_1}\cdot\cdot\cdot n_k^{\epsilon_k}=1} \prod_{\ell=1}^k\frac{\Lambda(n_\ell)}{\sqrt{n_\ell}}u_\ell
\big(\tfrac{\epsilon_\ell \log n_\ell}{H}\big) = S_{[k]}+ \sum_{\emptyset\subseteq J \subsetneq [k]} S_J\cdot O_k\bigg(\prod_{\ell\notin J}\frac{\|u_\ell\|_\infty + \|u_\ell'\|_\infty + \|u_\ell''\|_\infty}{H}\bigg),
\end{equation}
where $[k] = \{1,...,k\}$ and here for a set $J$ we define
$$
S_J = \sum \prod_\lambda \int_\mathbb{R}|x|u_{i_\lambda}(x)u_{j_\lambda}(-x)\,dx
$$
where the sum is over all partitions of $J$ into disjoint pairs $\{i_\lambda,j_\lambda\}$.
\end{lem}
Said another way,
$$
S_J = \sum_{\pi \in C(J)}\prod_{\ell\in J}\bigg(\int_\mathbb{R}|x|u_\ell(x)u_{\pi(\ell)}(-x)\,dx\bigg)^{1/2}
$$
where the set $C(J)$ is null for $|J|$ odd, and for $|J|$ even is the set of $(|J|-1)!!$ permutations of $J$ whose cycle type is of $|J|/2$ disjoint 2-cycles.

\begin{proof} By Lemma \ref{pnt}, for any $i,j$,
\begin{align}
\label{primesum}
\frac{1}{H^2}&\sum_{p_1^{\epsilon_1}p_2^{\epsilon_2}=1} \frac{\log p_1 \log p_2}{\sqrt{p_1 p_2}} u_i\Big(\tfrac{\epsilon_i \log p_i}{H}\Big)u_j\Big(\tfrac{\epsilon_2 \log p_2}{H}\Big) \notag \\
&= \int |x|u_i(x)u_j(-x)\,dx + O\bigg(\frac{\|u_i u_j\|_\infty + \|(u_i u_j)'\|_\infty + \|(u_i u_j)''\|_\infty}{H^2}\bigg) \notag \\
&= \int |x|u_i(x)u_j(-x)\,dx + O\bigg(\Big[\frac{\|u_i\|_\infty + \|u_i'\|_\infty + \|u_i''\|_\infty}{H}\Big]\Big[\frac{\|u_j\|_\infty + \|u_j'\|_\infty + \|u_j''\|_\infty}{H}\Big]\bigg),
\end{align}
where the initial sum is over all primes $p_1,p_2$ and signs $\epsilon_1,\epsilon_2$ with $p_1^{\epsilon_1}p_2^{\epsilon_2}=1$.

It follows that
$$
\frac{1}{H^k}\sum_{p^{\epsilon_1}\cdot\cdot\cdot p_k^{\epsilon_k}=1}\prod_{\ell=1}^k \frac{\log p_\ell}{\sqrt{p_\ell}}u_\ell\Big(\tfrac{\epsilon_\ell \log p_\ell}{H}\Big) = S_{[k]} + \sum_{\emptyset\subseteq J \subsetneq [k]} S_J\cdot O_k\bigg(\prod_{\ell\notin J}\frac{\|u_\ell\|_\infty + \|u_\ell'\|_\infty + \|u_\ell''\|_\infty}{H}\bigg),
$$
as, by the fundamental theorem of arithmetic, $p_1^{\epsilon_1}\cdot\cdot\cdot p_k^{\epsilon_k}=1$ if and only if primes match up pairwise $p_i=p_j$ with $\epsilon_i = -\epsilon_j$. The error term listed accumulates by expanding those products in which terms of the sort \eqref{primesum} occur.

It remains to show that
\begin{equation}
\label{5}
\frac{1}{H^k}\sum_{p_1^{\epsilon_1 \lambda_1}\cdot\cdot\cdot p_k^{\epsilon_k \lambda_k}=1}\prod_{\ell=1}^k\frac{\log p_\ell}{p_\ell^{\lambda_\ell/2}}u_\ell\Big(\tfrac{\epsilon_\ell \lambda_\ell \log p_\ell}{H}\Big) = \sum_{\emptyset\subseteq J \subsetneq [k]} S_J \cdot O_k\bigg(\prod_{\ell\notin J}\frac{\|u_\ell\|_\infty + \|u_\ell'\|_\infty + \|u_\ell''\|_\infty}{H}\bigg),
\end{equation}
where the sum is over primes $p_1,...,p_k$, signs $\epsilon_1,...,\epsilon_k$, and positive integers $(\lambda_1,...,\lambda_k) \in \mathbb{N}_+^k\setminus\{(1,1,...,1)\}.$ But the left hand side sum of \eqref{5} restricted to $\lambda$ with $\lambda_1\geq 3, ..., \lambda_k \geq 3$ is plainly
$$
O\bigg(\prod_{\ell = 1}^k \frac{\|u_\ell\|}{H}\bigg).
$$
On the other hand, for $\lambda$ with $\lambda_j$ fixed to equal $2$  for some $j$, by the fundamental theorem of arithmetic $p_1^{\epsilon_1 \lambda_1}\cdot\cdot\cdot p_k^{\epsilon_k \lambda_k}=1$ only in the case that $p_j = p_{j'}$ for some $j' \neq j$, so that thus restricted left hand side sum of \eqref{5} is
$$
\sum_{\substack{I \subset [k] \\ j\notin I}} O\bigg(\sum_{p_j} \frac{\log^2 p_j}{{p_j}^{3/2}}\cdot\prod_{\ell'\notin I}\frac{\|u_{\ell'}\|_\infty}{H} \times \frac{1}{H^{|I|}} \sum_p \prod_{\ell\in I} \frac{\log p_\ell}{p^{\lambda_\ell/2}}
u_\ell\Big(\tfrac{\epsilon_\ell \lambda_\ell \log p_\ell}{H}\Big)\bigg)
$$
where the sum with index labeled $p$ is over $p, \lambda, \epsilon$ such that $\prod_{\ell\in I} p_\ell^{\epsilon_\ell \lambda_\ell}=1$, and $I$ has the function in this sum of collecting those $p_i$ which are not equal to $p_j$. This expression is unpleasant, but our consolation is that it is only an error term. Applying it inductively, to bound the sums restricted to $\prod_{\ell \in I} p_\ell^{\epsilon_\ell \lambda_\ell} =1$, yields the Lemma. (We have here fixed $\lambda_j=2$, but of course to get an upper bound we need add at most $k$ sums like this.)
\end{proof}

As a consequence of Lemmas \ref{X7} and \ref{eval}, with $H = \tfrac{\log T}{n(T)}$,

\begin{cor}
\label{consequence}
For $\sigma$ as in Lemma \ref{X7}, and $\eta_1, ..., \eta_k$ such that $\mathrm{supp}\, \hat{\eta}_\ell \subset [-\delta_\ell H, \delta_\ell H] $, where $\delta_1+\cdots + \delta_k = \Delta < 2,$
\begin{align*}
\int_{\mathbb{R}}\frac{\sigma(t/T)}{T}&\prod_{\ell=1}^k \Bigg(-\!\!\!\!\!\!\!\int_{-\infty}^\infty\eta_\ell\big(\tfrac{\log T}{2\pi n(T)}(\xi_\ell-t)\big)dS(\xi_\ell)\Bigg)dt \\=\, & S_{[k]}+ O_k\bigg(\frac{1}{T^{1-\Delta/2}} \prod_{\ell=1}^k  \frac{\|\hat{\eta}_\ell\|_\infty}{\log T/n(T)} \bigg) + \sum_{\emptyset\subseteq J \subsetneq [k]} S_J\cdot O_k\bigg(\prod_{\ell\notin J}\frac{\|\hat{\eta}_\ell\|_\infty + \|\hat{\eta}_\ell'\|_\infty + \|\hat{\eta}_\ell''\|_\infty }{\log T / n(T)}\bigg),
\end{align*}
where $S_J$ is defined as in Lemma \ref{eval} with $u_\ell = \hat{\eta}_\ell$.
\end{cor}

\textit{Remark:} As an aside, we note that by modifying the above analysis, making $\Delta$ small enough, one can obtain an asymptotic even in the case that $n(T)$ grows like $O(T^{1-\delta})$, for $\delta >0$. In this case the result is less elegant, since the arithmetic factors present in Lemma \ref{X7} do not smooth out in the final asymptotic. We do not pursue these computations here, but they can be used to recover Fujii's macroscopic result, Theorem \ref{Fujiimacro}.

From Corollary \ref{consequence} it is an easy computation to see that

\begin{lem}
\label{9}
For $\eta$, $\sigma$ and $n(T)$ as in Theorem \ref{smoothedmeso}, with $\eta, \sigma,$ and $k$ fixed, and with $K$ a fixed continuous function supported in $(-1/k,1/k)$ such that $K(0)=1$,
\begin{equation}
\label{7}
\int_{\mathbb{R}}\frac{\sigma(t/T)}{T} \Bigg[-\!\!\!\!\!\!\!\int_{-\infty}^\infty\check{K}_{n(T)}
\ast\eta\big(\tfrac{\log T}{2\pi n(T)}(\xi-t)\big)dS(\xi)\Bigg]^k dt = (c_k+o(1))\Big[\int_{-n(T)}^{n(T)}|x||\hat{\eta}(x)|^2\,dx\Big]^{k/2}.
\end{equation}
\end{lem}

\begin{proof} Note that $[\check{K}_{n(T)}\ast\eta(\tfrac{\cdot}{n(T)})]\,\hat{}\,(\xi) = n(T)K(\xi)\hat{\eta}(n(T)\xi).$ By Corollary \ref{consequence}, for $K$ chosen to be supported in $(-1/k,1/k)$ we have the left hand side of \eqref{7} is
$$
(c_k+o(1))\bigg[\int_\mathbb{R}K^2\Big(\frac{x}{n(T)}\Big)|x|
\cdot|\hat{\eta}(x)|^2\, dx\bigg]^{k/2}.
$$
Because $\eta$ is of bounded variation, $\hat{\eta}(x) = O(1/x)$, and for any $c_1 > c_2 > 0$,
$$
\int_{c_1 n(T)}^{c_2 n(T)} |x| |\hat{\eta}(x)|^2 dx \lesssim \log(c_1/c_2) = o\bigg(\int_{-n(T)}^{n(T)}|x||\hat{\eta}(x)|^2 dx\bigg),
$$
since this latter integral diverges.\footnote{Even in the case it converges this $o$-bound is true, albeit for a different reason.} As we have that when $x \rightarrow 0$, $K^2(x) = 1 + o(1)$,
\begin{equation*}
\int_\mathbb{R}K^2\Big(\frac{x}{n(T)}\Big)|x|
\cdot|\hat{\eta}(x)|^2\, dx \sim \int_{-n(T)}^{n(T)}|x||\hat{\eta}(x)|^2\,dx. \qedhere
\end{equation*}
\end{proof}

\section{An Upper Bound}

We will be able to complete the proofs of our central limit theorems by showing that the left hand side of \eqref{7} is a good approximation to the left hand side of the equation in Theorem \ref{smoothedmeso}. We accomplish this mainly through the use of the following upper bound

\begin{thm}
\label{10}
For $\sigma$ as in Lemma \ref{X7},
\begin{equation}
\label{8}
\int_\mathbb{R} \frac{\sigma(t/T)}{T}\bigg[\int_{-\infty}^\infty \eta\big(\tfrac{\log T}{2\pi}(\xi-t)\big)\tilde{d}(\xi)d\xi\bigg]^k dt \lesssim_k \int_\mathbb{R} \frac{\sigma(t/T)}{T}\bigg[\int_{-\infty}^\infty M_k \eta\big(\tfrac{\log T}{2\pi}(\xi-t)\big) \log(|\xi|+2)\, d\xi\bigg]^k \,dt,
\end{equation}
with
$$
M_k \eta(\xi) = \sum_{\nu=-\infty}^\infty \sup_{I_k(\nu)}|\eta|\cdot\mathbf{1}_{I_k(\nu)}(\xi),
$$
where for typographical reasons we have denoted the interval $[k\nu - k/2, k\nu + k/2)$ by $I_k(\nu)$, and the order of our bound depends upon $k,\|\hat{\sigma}\|$ and the region in which $\hat{\sigma}$ can be supported.
\end{thm}

\begin{proof} We make use of the Fourier pair $V(\xi) = \big(\tfrac{\sin \pi \xi}{\pi \xi}\big)^2$ and $\hat{V}(x) = (1-|x|)_+$. Note that
$$
\eta(\xi) \lesssim \sum_\nu \sup_{I_k(\nu)}|\eta| \underbrace{V\big(\tfrac{\xi-\nu}{k}\big)}_{:=V_{\nu,k}(\xi)}.
$$
The right hand side of this is similar to $M_k \eta$ and we denote it by $M_k ' \eta$. What is important about the scaling is that $\hat{V}_{\nu,k}$ is supported in $(-1/k,1/k)$. Note that the left hand side of \eqref{8} is bounded by
\begin{align*}
&\lesssim \int_\mathbb{R} \frac{\sigma(t/T)}{T}\bigg[\int_{-\infty}^\infty M'_k\eta\big(\tfrac{\log T}{2\pi}(\xi-t)\big)\tilde{d}(\xi)d\xi\bigg]^k dt \\
&\lesssim [A^{1/k}+B^{1/k}]^k,
\end{align*}
where
\begin{align*}
&A = \int_\mathbb{R} \frac{\sigma(t/T)}{T}\bigg[\int_{-\infty}^\infty M'_k\eta\big(\tfrac{\log T}{2\pi}(\xi-t)\big)dS(\xi)\bigg]^k dt,\\
&B = \int_\mathbb{R} \frac{\sigma(t/T)}{T}\bigg[\int_{-\infty}^\infty M'_k\eta\big(\tfrac{\log T}{2\pi}(\xi-t)\big)\log(|\xi|+2)d\xi\bigg]^k dt,
\end{align*}
by Minkowski, and the fact that $\Omega(\xi)/2\pi = O\big(\log(|\xi|+2)\big)$.

By the restricted range of support for $\hat{V}_{\nu,l}$ and Lemmas \ref{X7} and \ref{eval}, for integers $\nu_1,...,\nu_k$
$$
\int_\mathbb{R}\frac{\sigma(t/T)}{T}\prod_{\ell=1}^k \bigg(\int_{-\infty}^\infty V_{\nu_\ell,k}\big(\tfrac{\log T}{2\pi}(\xi-t)\big) dS(\xi_\ell)\bigg) dt = O_k(1).
$$
Whence, taking a multilinear sum,
\begin{align*}
A &\lesssim_k \prod_{\ell=1}^k \sum_\nu \sup_{I_k(\nu)}|\eta|\\
&\lesssim B
\end{align*}
as $\log(|\xi|+2) \gtrsim 1$.

Finally,
$$
M'_k\eta(\xi) \lesssim \sum_{\mu=-\infty}^\infty \frac{1}{1+\mu^2} M_k \eta(\xi + \mu),
$$
so using $\log(|\xi+\mu|+2) \lesssim \log(|\xi|+2)\log(|\mu|+2)$,
$$
B \lesssim \int_\mathbb{R} \frac{\sigma(t/T)}{T}\bigg[\int_{-\infty}^\infty M_k \eta\big(\tfrac{\log T}{2\pi}(\xi-t)\big) \cdot\log(|\xi|+2)\, d\xi\bigg]^k \,dt.
$$
These estimates on $A$ and $B$ give us the result.
\end{proof}

This result should be viewed as a slight generalization of an $O_A(1)$ upper bound given by Fujii for the average number of zeros in an interval $[t,t+A/\log T]$ where $t$ ranges from $T$ to $2T$ \cite{Fu}.

\section{Proof of Theorems \ref{generalmeso} and \ref{generalmesoconv}}

We are now finally in a position to prove our main results. We first prove Theorem \ref{X}, then pass to Theorem \ref{smoothedmeso} (and hence to Theorem \ref{generalmeso}).

\begin{proof}[Proof of Theorem \ref{X}]
We want to show that
$$
E_T := \int_\mathbb{R}\frac{\sigma(t/T)}{T}\bigg[\int_\mathbb{R} \eta\big(\tfrac{\log T}{2\pi n(T)}(\xi-t)\big)dS(\xi)\bigg]^k - \Bigg[-\!\!\!\!\!\!\!\int_{-\infty}^\infty\check{K}_{n(T)}
\ast\eta\big(\tfrac{\log T}{2\pi n(T)}(\xi-t)\big)dS(\xi)\Bigg]^k dt
$$
is asymptotically negligible, where $K$ is a fixed function that meets the conditions of Lemma \ref{9}. In part because $k$ can be odd, we must use some care. To this end we have the following lemma.

\begin{lem}
\label{X11}
For $(X,d\mu)$ a positive measure space, $f$, $g$ real valued functions on $X$, and $k \geq 1$ an integer
$$
\bigg|\int (f^k-g^k) d\mu\bigg| \lesssim_k \|f-g\|_{L^k(d\mu)}\big(\|f\|_{L^k(d\mu)}^{k-1} + \|g\|_{L^k(d\mu)}^{k-1}\big).
$$
\end{lem}

\begin{proof}
If $f^k$ and $g^k$ are both almost everywhere the same sign, this is implied by Minkowski (with implied constant $k$). On the other hand, if $f^k$ and $g^k$ are almost always of opposite sign, the estimate is trivial. We can prove the lemma in general by breaking the integral over $X$ into two integrals over these subcases, and combine our estimates by noting that for positive $a$ and $b$, $a^\alpha+b^\alpha \leq 2\max(a^\alpha,b^\alpha)\lesssim (a+b)^\alpha$, where (in our case) $\alpha = (k-1)/k.$
\end{proof}

This leads us to consider
\begin{equation}
\label{i}
\int_\mathbb{R}\frac{\sigma(t/T)}{T}\Bigg[-\!\!\!\!\!\!\!\int_{-\infty}^\infty(\eta-\check{K}_{n(T)}
\ast\eta)\big(\tfrac{\log T}{2\pi n(T)}(\xi-t)\big)dS(\xi)\Bigg]^k dt.
\end{equation}
Trivally, this is bounded by
\begin{equation}
\label{i'}
\int_\mathbb{R}\frac{\sigma(t/T)}{T}\Bigg[\int_{-\infty}^\infty\Big|(\eta-\check{K}_{n(T)}
\ast\eta)\big(\tfrac{\log T}{2\pi n(T)}(\xi-t)\big)\Big| |dS|(\xi)\Bigg]^k dt,
\end{equation}
which by Theorem \ref{10} is bounded by
\begin{align*}
\lesssim& \int_\mathbb{R}\frac{\sigma(t/T)}{T}\Bigg[\int_{-\infty}^\infty M_{k/n(T)}(\eta-\check{K}_{n(T)}
\ast\eta)\big(\tfrac{\log T}{2\pi n(T)}(\xi-t)\big)\cdot \log(|\xi |+2)d\xi\Bigg]^k dt\\
=& \int_\mathbb{R}\frac{\sigma(t/T)}{T}\Bigg[\tfrac{2\pi n(T)}{\log T}\int_{-\infty}^\infty M_{k/n(T)}(\eta-\check{K}_{n(T)}
\ast\eta)(\xi)\log\big(\big|t+\tfrac{2\pi n(T) }{\log T}\xi\big|+2\big)d\xi\Bigg]^k dt\\
\lesssim& \bigg(\int_\mathbb{R}\frac{\sigma(t/T)}{T}\frac{\log^k(|t|+2)}{\log^k T} dt\bigg) \Bigg[n(T)\int_{-\infty}^\infty M_{k/n(T)}(\eta-\check{K}_{n(T)}
 \ast\eta)\big(\xi\big)d\xi\Bigg]^k \\
&\mathrm{      }+\Bigg[\frac{2\pi n(T)}{\log T}\int_{-\infty}^\infty M_{k/n(T)}(\eta-\check{K}_{n(T)}
 \ast\eta)\big(\xi\big)\log(|\xi|+2) d\xi\Bigg]^k.
\end{align*}
Note, if we label $L(\xi) = \log(|\xi|+2)$, we have $M_{k/n(T)}(\eta-\check{K}_{n(T)}
 \ast\eta)\big(\xi\big)\log(|\xi|+2) \leq M_{k/n(T)}\big[(\eta-\check{K}_{n(T)}
 \ast\eta)L\big]\big(\xi\big)$.

At this point we make use of the fact that $\eta$ is of bounded variation. Because $\eta$ has compact support,
$$
\int \log(|\xi|+2)|d\eta(\xi)| < +\infty.
$$
In addition, $\check{K}_{n(T)}\ast\eta$ is bounded in variation for the same reason that
$$
\int \log(|\xi|+2)\big|d\check{K}_{n(T)}\ast\eta(\xi)\big| = K(0)\int \log(|\xi|+2)|d\eta(\xi)| < +\infty.
$$
By the product rule then, $\mathrm{var}\big[(\eta-\check{K}_{n(T)}
 \ast\eta)L\big]$ is bounded, for $\mathrm{var}(\cdot)$ the total variation.

We have the following three lemmas:

\begin{lem}
\label{trunc}
For $f \in L^1(\mathbb{R})$ and of bounded variation $\mathrm{var}(f)$, and $K$ as above,
$$
\|f-\check{K}_H\ast f\|_{L^1} \lesssim \frac{\mathrm{var}(f)}{H}.
$$
\end{lem}

The proof is utterly standard, but I was unable to find a reference. The key point is that $K$ is smooth and compact, so that $|x||\check{K}(x)|$ is integrable.

\begin{proof}
Note that $\check{K}_H(x) = H \check{K}(Hx)$, so
\begin{align*}
\|f-\check{K}_H\ast f\|_{L^1} &= \bigg\|\int H \check{K}(H\tau)f(t)\,d\tau - \int H\check{K}(H\tau)f(t-\tau)\,d\tau\bigg\|_{L^1(dt)}\\
&\leq H\int \check{K}(H\tau)\|f(t)-f(t-\tau)\|_{L^1(dt)}d\tau\\
&\leq H\int \check{K}(H\tau)\Big(\int_\mathbb{R}\int_{-\tau}^0|df(t+h)|\,dh\,dt\Big)d\tau\\
&= H\int \check{K}(H\tau)|\tau|d\tau \cdot \mathrm{var}(f)\\
&\lesssim \frac{\mathrm{var}(f)}{H}.
\end{align*}
\end{proof}

Likewise, because $|\check{K}(x)||x|^2$ is integrable, and $|\check{K}(x)||x|\log(|x|+2)$ is of order $|\check{K}(x)||x|$ around $x=0$ and is bounded up to a constant by $|\check{K}(x)||x|^2$ otherwise, we have similarly,

\begin{lem}
\label{trunc2}
$$
\|f-\check{K}_H\ast f\|_{L^1(\log(|t|+2)dt)} \lesssim \frac{1}{H}\int_\mathbb{R}\log(|t|+2)|df(t)|.
$$
\end{lem}

Finally,
\begin{lem}
\label{maximal}
For $f$ of bounded variation, and any $\varepsilon > 0$,
$$
\sum_{k=-\infty}^\infty \varepsilon \|f\|_{L^\infty\big(\varepsilon[k-1/2,k+1/2)\big)} \lesssim \|f\|_{L^1} + \varepsilon\cdot \mathrm{var}(f) .
$$
\end{lem}

\begin{proof} For arbitrarily small $\varepsilon'$, we can choose $x_k \in \varepsilon[k-1/2,k+1/2)$ so that $|f(x_k)|$ is sufficiently close to $\|f\|_{L^\infty\big(\varepsilon[k-1/2,k+1/2)\big)}$ that
\begin{align*}
\sum_{k=-\infty}^\infty \varepsilon \|f\|_{L^\infty\big(\varepsilon[k-1/2,k+1/2)\big)} &\leq \varepsilon' + \varepsilon\sum_k |f(x_k)| \\
&\leq \varepsilon' + \sum_j (x_{2j+2}-x_{2j})|f(x_{2j})| + \sum_{j'}(x_{2j'+1}-x_{2j'-1})|f(x_{2j-1})|.
\end{align*}
More,
\begin{align*}
\bigg|\int |f| dx - \sum_j (x_{2j+2}-x_{2j})|f(x_{2j})|\bigg| &\leq \sum_j \int_{x_{2j}}^{x_{2j+2}}\Big||f(x)|-|f(x_{2j})|\Big|dx \\
&\leq \sum_j (x_{2j+2}-x_{2j})\int_{x_{2j}}^{x_{2j+2}} |df(x)|\\
&\leq 3\varepsilon\cdot \mathrm{var}(f)
\end{align*}
as $(x_{2j+2}-x_{2j}) \leq 3\varepsilon$ always. The same estimate holds for a sum over odd indices, and we have then
$$
\sum_k \varepsilon \|f\|_{L^\infty\big(\varepsilon[k-1/2,k+1/2)\big)} \leq \varepsilon' + 6\varepsilon\cdot \mathrm{var}(f) + 2 \int |f| dx.
$$
As $\varepsilon'$ was arbitrary, the lemma follows.
\end{proof}

Making use of these lemmas we have that
$$
\int_{-\infty}^\infty M_{k/n(T)}(\eta-\check{K}_{n(T)}
 \ast\eta)\big(\xi\big)d\xi \lesssim_{\eta,k} \frac{1}{n(T)},
$$
and
$$
\int_{-\infty}^\infty M_{k/n(T)}\big[(\eta-\check{K}_{n(T)}
 \ast\eta)\cdot L\big]\big(\xi\big)d\xi \lesssim_{\eta,k} \frac{1}{n(T)}.
$$

Hence \eqref{i} is bounded.  By Lemma \ref{X11}, with the averages over $t$ with respect to $\sigma$ playing the role of the positive measure $\mu$,
\begin{align}
\label{errorbound}
E_T \lesssim_{\eta,k}& \Bigg(\int_\mathbb{R}\frac{\sigma(t/T)}{T}\bigg|\int_{-\infty}^\infty \eta\big(\tfrac{\log T}{2\pi n(T)}(\xi-t)\big) dS(\xi)\bigg|^k dt\Bigg)^{(k-1)/k} \\
&+ \Bigg(\int_\mathbb{R}\frac{\sigma(t/T)}{T} \bigg|\,-\!\!\!\!\!\!\int_{-\infty}^\infty \check{K}_{n(T)}\ast\eta\big(\tfrac{\log T}{2\pi n(T)}(\xi-t)\big) dS(\xi)\bigg|^k dt\Bigg)^{(k-1)/k}.\notag
\end{align}
For $k$ even, this implies by Lemma \ref{9} (our Fourier truncation central limit theorem), and the fact that $\int |x||\hat{\eta}|^2 dx = +\infty$,
\begin{align*}
\int_\mathbb{R}\frac{\sigma(t/T)}{T}\bigg[\int_{-\infty}^\infty \eta\big(\tfrac{\log T}{2\pi n(T)}(\xi-t)\big) dS(\xi)\bigg]^k dt &= (c_k+o(1)) \Big[\int_{-n(T)}^{n(T)}|x||\hat{\eta}(x)|^2\,dx\Big]^{k/2} \\
&+ O\Bigg[\Bigg(\int_\mathbb{R}\frac{\sigma(t/T)}{T}\bigg|\int_{-\infty}^\infty \eta\big(\tfrac{\log T}{2\pi n(T)}(\xi-t)\big) dS(\xi)\bigg|^k dt\Bigg)^{(k-1)/k}\Bigg]
\end{align*}
This bound implies the left hand side diverges, and thus the conclusion of Theorem \ref{X} for even $k$. For odd $k$, by H\"{o}lder (or Cauchy-Schwartz) and the result we have just proved for even $k$,
\begin{equation}
\label{oddbound}
\int_\mathbb{R}\frac{\sigma(t/T)}{T}\bigg|\int_{-\infty}^\infty \eta\big(\tfrac{\log T}{2\pi n(T)}(\xi-t)\big) dS(\xi)\bigg|^k dt \leq (\sqrt{c_{2k}}+o(1)) \Big[\int_{-n(T)}^{n(T)}|x||\hat{\eta}(x)|^2\,dx\Big]^{k/2},
\end{equation}
and hence, using \eqref{errorbound} again, Theorem \ref{X} for odd $k$ as well.
\end{proof}

\begin{proof}[Proof of Theorem \ref{smoothedmeso}]
To see that Theorem \ref{X} implies Theorem \ref{smoothedmeso}, note that for any $\epsilon > 0$, we can find $\sigma_1$ of the sort delimited in Theorem \ref{X}, so that $\|\mathbf{1}_{[1,2]}-\sigma_1\|_{L^1} < \epsilon/2$. Further, we can find $\sigma_2$, a linear combination of translations and dilations of the function $\big(\tfrac{\sin \pi t}{\pi t}\big)^2$, so that $\sigma_2$ is non-negative and $|\mathbf{1}_{[1,2]}(t)-\sigma_1(t)|\leq \sigma_2(t)$ for all $t$, and $\|\sigma_2\|_{L^1} < \epsilon$. Note (for simplicity of notation) that \eqref{oddbound} is true for even $k$ as well, and by rescaling linearly, we have
$$
\int_\mathbb{R}\frac{\sigma_2(t/T)}{T}\bigg|\int_{-\infty}^\infty \eta\big(\tfrac{\log T}{2\pi n(T)}(\xi-t)\big) dS(\xi)\bigg|^k dt \leq \epsilon(\sqrt{c_{2k}}+o(1)) \Big[\int_{-n(T)}^{n(T)}|x||\hat{\eta}(x)|^2\,dx\Big]^{k/2}.
$$
Then
$$
\int_\mathbb{R}\frac{\mathbf{1}_{[1,2]}(t/T)}{T}\bigg[\int_{-\infty}^\infty \eta\big(\tfrac{\log T}{2\pi n(T)}(\xi-t)\big) dS(\xi)\bigg]^k dt = [c_k + o(1) + \epsilon\cdot(O_k(1)+o(1))] \Big[\int_{-n(T)}^{n(T)}|x||\hat{\eta}(x)|^2\,dx\Big]^{k/2}.
$$
(Note that here the $O_k(1)$ term is bounded absolutely by $\sqrt{c_{2k}}$.) As $\epsilon$ is arbitrary, the theorem follows.
\end{proof}

\begin{proof}[Proof of Theorem \ref{generalmesoconv}]
A proof will follow almost exactly as before. We need only to show that Theorem \ref{smoothedmeso} is true for $\eta$ instead of the sort delimited in Theorem \ref{generalmesoconv}. The reader may check that the only part of the proof which requires modification is that the error term $E_T$, at the start of section 4, cannot be shown to be asymptotically negligible in the same way as before, since now asymptotically negligible means that $E_T = o(1)$. But using Lemma \ref{X11} in the same way as before, this will be the case, and therefore the theorem, so long as
\begin{equation}
\label{fourierdecay}
\|\eta-\check{K}_H\ast\eta\|_{L^1} = o(1/H),
\end{equation}
for some $K$ as above. This is certainly the case for those $\eta$ delimited in Theorem \ref{generalmeso}, using the fact that for such $\eta$, $\hat{\eta}(\xi) =  o(1/(1+|\xi|)^2)$.
\end{proof}

\textit{Remark:} \eqref{fourierdecay} is true for a wider range of functions than $\mathit{C}^2_{c}(\mathbb{R})$; but it does not encompass the elegant criterion, ``all functions which are of bounded variation and compactly supported." It is \emph{not} the case for $\eta$ a Cantor function, for instance. We expect the theorem to remain true in this case, but to prove this would seem to require upper bounds on correlation functions for zeta zeroes with respect to oscillatory functions, extending outside the range of functions considered by Rudnick and Sarnak. Although here we require only upper bounds, not exact evaluations, this still goes beyond what we currently seem able to prove.

\section{Acknowledgements}

I would like to thank Rowan Killip, Zeev Rudnick, and Terence Tao for a number of helpful and encouraging exchanges, and additionally anonymous referees for helpful comments and criticism.

\section{Appendix: Towards a Mesoscopic Theory}

We include in this appendix a more general discussion of the statistics of the zeroes of the zeta function in the mesoscopic regime. Our discussion will culminate in Theorem \ref{a6}, a statement from which one can deduce both the microscopic linear statistics of the sort considered by Rudnick and Sarnak and the central limit theorems discussed above, along with covariance statements for translated linear statistics separated by mesoscopic distances. Other theorems concerning the mesoscopic distribution of zeta zeroes, which also depend upon the macroscopic statistics of the zeroes, can be found in \cite{Bou} and \cite{Ka}.

To motivate what follows, we want to show first that Corollary \ref{consequence} implies the well-known result of Rudnick and Sarnak that, upon ordering the positive ordinates of the zeroes $0 < \gamma_1 \leq \gamma_2 \leq ...$,

\begin{thm}[Rudnick-Sarnak]
\label{a1}
For $\eta: \mathbb{R}^k \rightarrow \mathbb{R}$ such that $\mathrm{supp}\, \hat{\eta} \subseteq \{x\in \mathbb{R}^k: |x_1| + \cdot\cdot\cdot + |x_k| < 2\}$,
$$
\lim_{T\rightarrow\infty} \frac{1}{T}\int_T^{2T}\sum_{\substack{i_1,...,i_k\\\mathrm{distinct}}}\eta\Big(\tfrac{\log T}{2\pi}(\gamma_{i_1}-t),...,\tfrac{\log T}{2\pi}(\gamma_{i_k}-t)\Big)\,dt = \int_{\mathbb{R}^k}\eta(x) \det_{k\times k}[S(x_i-x_j)]\, d^k x,
$$
where $S(\xi) = \tfrac{\sin \pi \xi}{\pi \xi}$ and $\det_{k\times k}[S(\xi_i-\xi_j)]$ is a $k\times k$ determinant with $ij^{\textrm{th}}$ entry $S(\xi_i-\xi_j)$.
\end{thm}

That is to say, with respect to sufficiently smooth functions, the zeroes of the zeta function tend weakly to a determinantal point process with sine-kernel.

One may do this either through a combinatorial sieving procedure -- effectively this is the proof of Rudnick and Sarnak -- or alternatively one may use the combinatorics of Diaconis and Shahshahani. For us, it will be more enlightening to use the latter. Proceeding in this manner originated with Hughes and Rudnick, although our range of test functions will coincide with the slightly wider range used originally by Rudnick and Sarnak.

The theorem of Diaconis and Shahshahani we will need is

\begin{thm}[Diaconis-Shahshahani]
\label{a2}
Let $\mathcal{U}(n)$ be the set of $n\times n$ unitary matrices endowed with Haar measure. Consider $a = (a_1,...,a_k)$ and $b = (b_1,...,b_k)$ with $a_1,a_2,...,b_1,b_2,... \in \{0,1,...\}$.If $\sum_{j=1}^k j a_j \neq \sum_{j=1}^k j b_j$,
\begin{equation}
\label{aa1a}
\int_{\mathcal{U}(n)}\prod_{j=1}^k \Tr(g^j)^{a_j}\overline{\Tr(g^j)^{b_j}}\,dg = 0.
\end{equation}
Furthermore, in the case that
$$
\max\Bigg(\sum_{j=1}^k j a_j, \sum_{j=1}^k j b_j\Bigg) \leq n
$$
we have
\begin{equation}
\label{aa1}
\int_{\mathcal{U}(n)}\prod_{j=1}^k \Tr(g^j)^{a_j}\overline{\Tr(g^j)^{b_j}}\,dg = \delta_{ab}\prod_{j=1}^k j^{a_j} a_j!
\end{equation}
\end{thm}
In addition, for unrestricted $a$
$$
\int_{\mathcal{U}(n)}\prod_{j=1}^k\big|\Tr(g^j)\big|^{2a_j}\,dg \leq \prod_{j=1}^k j^{a_j}a_j!
$$
but we will not need this fact. In general, for products of traces outside of the restricted range of the theorem, no pattern emerges which is as nice as \eqref{aa1}. Since our restricted range here corresponds -- as we will show shortly -- to the only range of test functions for which the statistics of the zeta function's zeroes can be rigorously evaluated, this fact must be seen as somewhat curious.

Here trace is defined in the standard way, so that $\Tr(I_{n\times n}) = n$. For a proof of Theorem \ref{a2}, see \cite{DiEv} or \cite{Bu}.

It is a simple exercise in enumerative combinatorics to see that \eqref{aa1a} and \eqref{aa1} imply that for $|j_1|+\cdot\cdot\cdot+|j_k| \leq 2n$
$$
\int_{\mathcal{U}(n)}\prod_{\ell=1}^k \Tr(g^{j_\ell})\,dg = \sum \prod_\lambda |j_{\mu_\lambda}|\, \delta(j_{\mu_\lambda}=-j_{\nu_\lambda})
$$
where once again the sum is over all partitions of $[k]$ into disjoint pairs $\{\mu_\lambda,\nu_\lambda\}$, and $\delta(j_{\mu_\lambda}=-j_{\nu_\lambda})$ is $1$ or $0$ depending upon whether $j_{\mu_\lambda}=-j_{\nu_\lambda}$ or not.

We are able to use this to study the determinantal point process with sine kernel because the eigenvalues of a random unitary matrix, properly spaced, are themselves a determinantal point process with kernel tending to that of the sine kernel. This is due, in effect, to Weyl.

\begin{prop}
\label{a3}
Let $\{e(\theta_1), e(\theta_2),...,e(\theta_n)\}$ be the eigenvalues of a random unitary matrix, distributed according to Haar measure, with $\theta_i \in [-1/2,1/2)$ for all $i$. Then the points $\{n\theta_1,...,n\theta_n\}$ comprise a determinantal point process $\mathcal{S}_n$ on $[-n/2,m/2)$ with kernel in $x,y$ given by $S_n(x-y) = \tfrac{\sin \pi (x-y)}{n \sin (\pi (x-y)/n)}$. That is for any test function $\eta$,
\begin{align*}
\E_{\mathcal{S}_n} \sum_{\substack{i_1,...,i_k\\\mathrm{distinct}}} \eta(\xi_{i_1},...,\xi_{i_k}) &= \int_{\mathcal{U}(n)}\sum_{\substack{i_1,...,i_k\\\mathrm{distinct}}} \eta(n\theta_{i_1},...,n\theta_{i_k}) \,dg \\
&=\int_{[-n/2,n/2]^k}\eta(x_1,...,x_k) \det_{k\times k}[S_n(x_i-x_j)]\,d^k x
\end{align*}
\end{prop}
For further discussion see \cite{Co}.

We use this to prove

\begin{thm}
\label{a4}
If $\mathcal{S}$ is the determinantal point process with kernel in $x,y$ given by $S(x-y)$ for $S(x) = \tfrac{\sin \pi x}{\pi x}$, then for functions $\eta_1, ..., \eta_k$ such that, as in Lemma \ref{X7}, $\mathrm{supp }\, \eta_\ell \in [-\delta_\ell,\delta_\ell]$ with $\delta_1 + \cdot\cdot\cdot + \delta_k \leq 2$,
\begin{equation}
\label{aeq1}
\E_\mathcal{S}\prod_{\ell=1}^k\Big(\Delta_{\eta_\ell}-\E_\mathcal{S}\Delta_{\eta_\ell}\Big) = S_{[k]}
\end{equation}
where $S_{[k]}$ is defined as in Corollary \ref{consequence}, and here $\Delta_\eta = \sum \eta(\xi_i)$ as before, for $\{\xi_i\}$ the points of the process.
\end{thm}

Note that here, by definition, $\E \Delta_\eta = \int \eta\,dx$.

Before we come to the proof, we note that as an easy consequence, upon expanding the product in \ref{aeq1} and applying induction,

\begin{cor}
\label{a5}
A point process $\mathcal{P}$ satisfies \eqref{aeq1} for all $k$ over the range of test functions restricted as in Theorem \ref{a4} if and only if for all $k$ and for any integrable $\eta$ defined on $\mathbb{R}^k$ with $\mathrm{supp }\,\hat{eta} \subseteq \{y \in \mathbb{R}^k: |y_1| + \cdot\cdot\cdot + |y_k| \leq 2\}$,
$$
\E_\mathcal{S} \sum_{\substack{i_1,...,i_k\\\mathrm{distinct}}}\eta(\xi_{i_1},...,\xi_{i_k}) = \int_{\mathbb{R}^k} \eta(x_1,...,x_k) \det_{k \times k}[ S(x_i -x_j) ] d^k x.
$$
\end{cor}

\begin{proof}[Proof of Theorem \ref{a4}]
For a function $\eta$, define
$$
\eta^{(n)}(\theta) = \sum_{k\in \mathbb{Z}} \eta(\theta + nk).
$$
Note that for Schwartz $\eta$, $\eta^{(n)} \rightarrow \eta$ uniformly. We have then that for fixed Schwartz $\eta_1, ..., \eta_k$,
\begin{align*}
\E_\mathcal{S}\prod_{\ell=1}^k\Big(\Delta_{\eta_\ell}-\E_\mathcal{S} \Delta_{\eta_\ell}\Big) &= \lim_{n\rightarrow\infty} \E_{\mathcal{S}_n}\prod_{\ell=1}^k \Big(\Delta_{\eta_l^{(n)}}-\E_{\mathcal{S}_n}\Delta_{\eta_l^{(n)}}\Big) \\
&= \lim_{n\rightarrow\infty} \int_{\mathcal{U}(n)} \prod_{\ell=1}^k\bigg(\sum_{\nu=1}^n \eta_\ell^{(n)}(n\theta_\nu)-n\int_{-1/2}^{1/2} \eta_\ell^{(n)}(n\theta)\,d\theta\bigg) \, dg
\end{align*}
But by Poisson summation,
$$
\eta_\ell^{(n)}(n\theta_\nu)-\int_{-1/2}^{1/2}\eta_\ell^{(n)}(n\theta)\,d\theta = \sum_{j\in \mathbb{Z}\backslash \{0\}} \frac{1}{n} \hat{\eta}_\ell\Big(\frac{j}{n}\Big) e(j\theta),
$$
so that
\begin{align*}
\int_{\mathcal{U}(n)} \prod_{\ell=1}^k\bigg(\sum_{\nu=1}^n \eta_\ell^{(n)}(n\theta_\nu)-n\int_{-1/2}^{1/2} \eta_\ell^{(n)}(n\theta)\,d\theta\bigg) \, dg
&= \int_{\mathcal{U}(n)}\prod_{\ell=1}^k \sum_{j \in \mathbb{Z}\backslash\{0\}} \frac{1}{n} \hat{\eta}_\ell\Big(\frac{j}{n}\Big)\Tr(g^{j})\,dg\\
&= \sum_{j_1,...,j_k \in \mathbb{Z}\backslash\{0\}}\prod_{\ell=1}^k\frac{1}{n} \hat{\eta}_\ell\Big(\frac{j_\ell}{n}\Big) \cdot \int_{\mathcal{U}(n)}\prod_{\ell=1}^k \Tr(g^{j_\ell})\, dg.
\end{align*}
But for $\hat{\eta}_1, ..., \hat{\eta}_k$ restricted as in the Theorem, this sum is only over those $j$ with $|\tfrac{j_1}{n}| + \cdot\cdot\cdot + |\tfrac{j_k}{n}| \leq 2$. In this case the above sum reduces to
$$
\sum\prod_\lambda\bigg(\sum_{j\in\mathbb{Z}\backslash\{0\}}\frac{1}{n}\frac{|j|}{n} \hat{\eta}_{\mu_\lambda}\Big(\frac{j}{n}\Big)\hat{\eta}_{\nu_\lambda}\Big(\frac{-j}{n}\Big)\bigg).
$$
Clearly this tends to $S_{[k]}.$
\end{proof}

\begin{proof}[Proof of Theorem \ref{a1}]
Using Corollary \ref{consequence} for $n(T)=1$, we have for $\eta_1,...,\eta_\ell$ as in Theorem $\ref{a4}$,
$$
\lim_{T\rightarrow\infty} \int_\mathbb{R}\frac{\sigma(t/T)}{T}\prod_{\ell=1}^k \int_{-\infty}^\infty \eta\Big(\tfrac{\log T}{2\pi}(\xi_\ell-t)\Big) \,dS(\xi_\ell)\,dt = S_{[k]}.
$$
But by Stirling's formula,
$$
\int_{-\infty}^\infty \eta\Big(\tfrac{\log T}{2\pi}(\xi_\ell-t)\Big)\, dS(\xi) = \sum_\gamma \eta\Big(\tfrac{\log T}{2\pi}(\gamma-t)\Big) - \int \eta(x)\, dx + o(1).
$$
Expanding the product as in Corollary \ref{a5}, and passing from $\sigma$ to $\mathbf{1}_{[1,2]}$ as before yeilds the claim for $\eta = \eta_1\otimes\cdot\cdot\cdot\otimes\eta_k$. We can pass to general $\eta$ by uniformly approximating such $\eta$ and using Theorem \ref{10} to bound the difference between the linear statistics of $\eta$ and those of its approximation.
\end{proof}

The convergence here is microscopic, and therefore cannot, unless spread over a wider region as in Corollary \ref{a4}, yield a mesoscopic central limit theorem like Fujii's or Theorem \ref{generalmeso}. In a general way, it does appear that in the mesoscopic regime, the zeroes of the zeta functions are spaced like the points of a sine-kernel determinantal point process -- and that moreover we have knowledge of this fact as long as any test functions used remain microscopically band-limited. Stating this principle in a way which is both (i) precise, and (ii) satisfying, is a rough task however. We shall make an attempt below, but we should be forthright that it is only the first of these conditions and not the second that is really achieved. Before proceeding, it is worthwhile to discuss the matter heuristically somewhat further.

We say that a point process $\mathcal{P}$ is ``mock-determinantal with sine-kernel" if its correlation functions agree with that of $\mathcal{S}$ with respect to sufficiently smooth test functions; that is
$$
\E_\mathcal{P}\sum_{\substack{i_1,...,i_k\\\mathrm{distinct}}}\eta(\xi_{i_1},...,\xi_{i_k}) = \E_\mathcal{S} \sum_{\substack{i_1,...,i_k\\\mathrm{distinct}}}\eta(\xi_{i_1},...,\xi_{i_k})
$$
with respect to -- say for our purposes -- $\eta$ with Fourier transform $\hat{\eta}$ supported on $\{x\in \mathbb{R}^k: |x_1|+\cdot\cdot\cdot+|x_k| \leq 2\}.$ Using the proof above for the zeroes of the zeta function, one can show that Theorems \ref{generalmeso} and \ref{generalmesoconv} hold for any such $\mathcal{P}$. That is for $\eta$ restricted as in either theorem, a parameter $L$ which grows, and $\Delta_\eta = \sum \eta(\xi_i/L)$,
$$
\frac{\Delta_\eta-\E\Delta_\eta}{\sqrt{\Var \Delta_\eta}} \Rightarrow N(0,1),
$$
as $L\rightarrow\infty$. (As here we are dealing with a single point process $\mathcal{P}$, `mesoscopic' restrictions on the growth of $L$ play no role.) We may ask whether there exists any such mock-determinantal point processes $\mathcal{P}$ for which $\eta$ is of bounded variation, but $(\Delta_\eta-\E\Delta_\eta)/\sqrt{\Var \Delta_\eta}$ does not tend to the normal distribution. I do not know the answer to this, but I suspect that there does. This would imply that to fill the small gap between Theorems \ref{generalmeso} and \ref{generalmesoconv} and their random matrix analogues will require (a small amount of) statistical information about the zeroes of the zeta function outside of that provided by test functions which are band-limited as in Rudnick-Sarnak.

We return to our goal of characterizing the zeroes of the zeta function in the mesoscopic regime in a way that retains microscopic statistics as well. We have:

\begin{thm}
\label{a6}
Let $\sigma$ be as in Theorem \ref{X}, and let $Z_T$ be the point process defined by the points $\{\tfrac{\log T}{2\pi}(\gamma-t)\}$ where $\gamma$ runs through the ordinates of zeroes of the zeta function, and $t$ is a random variable in $\mathbb{R}$ with distribution given by $\sigma(t/T)/T$. For fixed $A < 2$, fixed $r$ of compact Fourier support, and fixed $n(T)$ with $n(T)\rightarrow\infty$ but with $n(T) = o(\log T)$, we have that for $|\alpha_1| + \cdot\cdot\cdot + |\alpha_k| \leq A$,
$$
\E_{Z_T}\prod_{\ell=1}^k\Big(\Delta_\ell-\E \Delta_\ell\Big) = \E_\mathcal{S}\prod_{\ell=1}^k\Big(\Delta_\ell-\E \Delta_\ell\Big) + \sum_{\emptyset\subseteq J \subsetneq [k]} \varepsilon([k]\backslash J)\cdot\E_\mathcal{S}\prod_{\ell\in J}\Big(\Delta_\ell-\E \Delta_\ell\Big),
$$
where
$$
\Delta_\ell = \sum r\Big(\frac{\xi_i}{n(T)}\Big)e(\alpha\xi_i)
$$
for the terms $\varepsilon([k]\backslash J)$ having no dependence on $\alpha_i$ with $i\in J$, and tending to $0$ uniformly as $T\rightarrow\infty$.
\end{thm}

This may be proven by following exactly the proof of Corollary \ref{consequence}. By slightly modifying the proof, one may prove this theorem even for $\sigma = \mathbf{1}_{[1,2]}$ so that $t$ is uniformly distributed between $T$ and $2T$, but we do not pursue this matter here. By integrating in $\alpha$, one can obtain microscopic and macroscopic statistics, and correlations thereof, uniformly for points separated by a distance asymptotically less than $m(T)$. One can, for instance, recover Corollary \ref{consequence} for $n(T) = o(m(T))$ in this way. We are able to integrate in $\alpha$ without destroying error terms for the reason that $\varepsilon([k]\backslash J)$ has no dependence on $\alpha_i$ for $i\in J$.

In the same way, by modifying the proof of Theorem \ref{a4},

\begin{thm}
\label{a7}
For fixed $A < 2$, fixed $r$ of compact Fourier support, and fixed $n(N)$ with $n(N)\rightarrow\infty$ but with $n(N) = o(N)$, we have that for $|\alpha_1| + \cdot\cdot\cdot + |\alpha_k| \leq A$,
$$
\E_{\mathcal{S}_N}\prod_{\ell=1}^k\Big(\Delta_\ell-\E \Delta_\ell\Big) = \E_\mathcal{S}\prod_{\ell=1}^k\Big(\Delta_\ell-\E \Delta_\ell\Big) + \sum_{\emptyset\subseteq J \subsetneq [k]} \varepsilon([k]\backslash J)\cdot\E_\mathcal{S}\prod_{\ell\in J}\Big(\Delta_\ell-\E \Delta_\ell\Big),
$$
for $\Delta_\ell$ (defined in the obvious way with respect to $n(N)$), and $\varepsilon$ as in Theorem \ref{a6}.
\end{thm}

To have a more eloquent expression of the mesoscopic convergence expressed by these results would certainly be desirable.

We want finally to point out again that Selberg's approximation to $S(t)$, mentioned in the introduction, and therefore Fujii's Theorem's \ref{Fujiimeso} and \ref{Fujiimacro}, are true unconditionally. The first of these claims was shown by Selberg, using a zero-density estimate to bound the number of zeroes lying off the critical line. I have been unable to extend this method to prove Theorem \ref{generalmeso} unconditionally, where the points we are counting are the imaginary ordinates of non-trivial zeroes -- zeroes which may in some instances lie off the critical line -- and I leave it as a challenge for readers to do so.

\end{document}